\documentclass[]{article}
\usepackage[utf8]{inputenc}
\usepackage{amsmath,amssymb,amsthm,mathtools}
\usepackage{graphicx}
\usepackage[mathscr]{euscript}

\newtheorem{thm}{Theorem}

\newtheorem{cor}{Corollary}
\newtheorem{lem}{Lemma}
\newtheorem{prop}{Proposition}

\newtheorem*{thm*}{Theorem}

\theoremstyle{definition}
\newtheorem{defn}{Definition}

\title{On regular families of cardinal interpolators and multiresolution analyses}
\author{Jeff Ledford}
\date{June 2014}

\begin{document}

\maketitle

\begin{abstract}
In this short note, we investigate the relationship between so-called regular families of cardinal interpolators and multiresolution analyses.  We focus our studies on examples of regular families of cardinal interpolators whose Fourier transform is unbounded at the origin.  In particular, we show that when this is the case there is a multiresolution analysis corresponding to each member of a regular family of cardinal interpolators.  

\end{abstract}


\section{Introduction}
This paper grew out of exploring the connections between polyharmonic splines and multiquadrics.  Some of this connection has been detailed in \cite{me}, where similar $L^2$ convergence properties are exhibited for the associated fundamental functions of interpolation.  A result concerning general $L^p$, $1<p<\infty$ results may be found in \cite{me2}.  In \cite{Madych}, it was shown that polyharmonic splines may be used to generate multiresolution analyses.  We show that the same is true for some examples of regular families of cardinal interpolators.

This paper is organized as follows, th next section contains definitions and basic facts, while in Section 3 we construct multiresolution analyses.  We collect examples and remarks in Section 4.


\section{Preliminaries}

The purpose of this section is to present the reader with several definitions and basic facts which will be used to prove the main result.  Since our calculations frequently occur in the frequency domain, we begin with the Fourier transform.

\begin{defn}
For a function $f\in C^{\infty}(\mathbb{R}^n)$, and multi-indices $\alpha, \beta$, let
\[
\| f \|_{\alpha,\beta} = \sup_{x\in \mathbb{R}^n}|x^\alpha D^\beta f(x)  |.
\]
We define the \emph{Schwartz space}, denoted $\mathscr{S}$, by
\[
\mathscr{S}=\left\{ f\in C^{\infty}(\mathbb{R}^n): \| f \|_{\alpha,\beta}<\infty \text{ for all } \alpha,\beta  \right\}.
\]
\end{defn}

\begin{defn} If $f\in \mathscr{S}$, we define its \emph{Fourier transform}, denoted $\hat{f}(\xi)$, to be:
\begin{equation}\label{FT definition}
\hat{f}(\xi)=\int_{\mathbb{R}^n}f(x)e^{-i\langle \xi,x \rangle}dx,
\end{equation}
where $\langle \xi,x\rangle=\sum_{j=1}^{n}\xi_j x_j$.
\end{defn}

We may extend the Fourier transform to the dual of $\mathscr{S}$, denoted $\mathscr{S}'$, in the usual way.  That is, if $f\in \mathscr{S}'$, then $\hat f$ satisfies $\langle \hat f, \varphi \rangle = \langle f, \hat\varphi  \rangle$ for all $\varphi \in \mathscr{S}$.  If $f$ is a function of at most polynomial growth we say $f$ is \emph{slowly increasing}, and note that $f$ may be identified with $f\in\mathscr{S}'$ by the formula $\langle f, \varphi \rangle = \int_{\mathbb{R}^n} \varphi(x) f(x)dx$, where $dx$ is the Lebesgue measure on $\mathbb{R}^n$. 

On the Hilbert space $L^2(\mathbb{R}^n)$, this convention leads to the isometry $f\mapsto (2\pi)^{-n/2}\hat f$.  Plancherel's theorem provides us with the inner product identity 
\begin{equation}\label{inner product}
\langle f, g \rangle =  (2\pi)^{-n}\langle \hat f , \hat g \rangle ,
\end{equation}
where $\langle f ,g \rangle =\int_{\mathbb{R}^n}f(x)\overline{g(x)}dx$. 

Our results concern functions known as cardinal interpolators and related concepts.  These were introduced in \cite{me} and details may be found there.

\begin{defn}  We say that a function $\phi$ is a \emph{cardinal interpolator} if it satisfies the following conditions:
\begin{enumerate}
\item[(H1)] $\phi$ is a real valued slowly increasing function on $\mathbb{R}^n$;
\item[(H2)] $\hat{{\phi}}(\xi) \geq 0$ and $\hat{{\phi}}(\xi)\geq \delta >0$ in $[-\pi,\pi]^n$;
\item[(H3)] $\hat{{\phi}}\in C^{n+1}(\mathbb{R}^n\setminus\{0\})$;
\item[(H4)] There exists $\epsilon >0$ such that if $|\alpha|\leq n+1$,  $D^{\alpha}\hat\phi(\xi)=O(\| \xi \|^{-(n+\epsilon)})$ as $\|\xi\|\to\infty$;
\item[(H5)] for any multi-index $\alpha$, with $|\alpha|\leq n+1$, 
\[
\dfrac{\displaystyle\prod_{j=1}^{|\alpha|}D^{\alpha_j}\hat\phi  }     {\hat\phi^{|\alpha|+1}} \in L^{\infty}([-\pi,\pi]^n), \quad\text{where}\quad \sum_{j=1}^{|\alpha|}\alpha_j = \alpha.
\] 
\end{enumerate}
\end{defn}

\begin{defn}
Given a cardinal interpolator $\phi$, we define the \emph{fundamental function} of interpolation, denoted $L_\phi$, by its Fourier transform
\begin{equation}\label{fundamental function}
\hat{L}_\phi(\xi)=\hat{\phi}(\xi)\left[\sum_{j\in\mathbb{Z}^n}\hat{\phi}(\xi-2\pi j)\right]^{-1}.
\end{equation}
\end{defn}

\begin{defn}
We call a family of functions $\{\phi_{\alpha}:\alpha\in A\}$ a \emph{regular family of cardinal interpolators} if for each $\alpha\in A\subset(0,\infty)$, $\phi_{\alpha}$ is a cardinal interpolator and in addition to this, we have:
\begin{enumerate}
\item[(R1)] for $j\in\mathbb{Z}^n\setminus\{0\}$ and $\xi\in[-\pi,\pi]^n$, define $M_{j,\alpha}(\xi)=\displaystyle\dfrac{\hat{\phi}_\alpha(\xi+2\pi j)}{\hat{\phi}_\alpha(\xi)} $, then $\displaystyle\lim_{\alpha\to\infty}M_{j,\alpha}(\xi)=0$ for almost every $\xi\in[-\pi,\pi]^n$;
\item[(R2)] there exists $\{M_j\}\in l^1(\mathbb{Z}^n\setminus\{0\})$, independent of $\alpha$, such that for all $j\in\mathbb{Z}^n\setminus\{0\}$, $M_{j,\alpha}(\xi)\leq M_j$ for almost every $\xi\in[-\pi,\pi]^n$.
\end{enumerate}
\end{defn}
 
The following lemma is a straightforward consequence of these definitions.

\begin{lem}
If $\{\phi_\alpha:\alpha\in A\}$ is a regular family of cardinal interpolators, then 
\begin{equation}\label{bound}
 1\leq \left[{\displaystyle\sum_{j\in\mathbb{Z}^n}\hat{\phi}_\alpha(\xi-2\pi j)     }\right]{\left[\displaystyle\sum_{j\in\mathbb{Z}^n}\hat{\phi}^2_\alpha(\xi-2\pi j) \right]^{-1/2}} \leq 1+C 
\end{equation}
for all $\xi\in [-\pi,\pi]^n$, where $C>0$ is independent of $\alpha\in A$.
\end{lem}

\begin{proof}
The lower bound is a consequence of (H2).  To see the upper bound, for $\xi\in [-\pi, \pi]^n$ we write
\[
\displaystyle\sum_{j\in\mathbb{Z}^n}\hat{\phi}_\alpha(\xi-2\pi j)  =\hat\phi_\alpha(\xi) + u_\alpha(\xi) ,
\]
where $u_\alpha(\xi)=\sum_{j\neq 0}\hat{\phi}_\alpha(\xi-2\pi j) $.  Using (H2) and (R2), we have
\begin{align*}
&\left[{\displaystyle\sum_{j\in\mathbb{Z}^n}\hat{\phi}_\alpha(\xi-2\pi j)     }\right]{\left[\displaystyle\sum_{j\in\mathbb{Z}^n}\hat{\phi}^2_\alpha(\xi-2\pi j) \right]^{-1/2}} \\
\leq &\left[ 1+ 2u_\alpha(\xi)/\hat\phi_\alpha(\xi) + (u_\alpha(\xi)/\hat\phi_\alpha(\xi))^2  \right]^{1/2} \\
\leq & \left[1+2\|\{M_j\}\|_{l^1}+\|\{M_j\}\|_{l^1}^2\right]^{1/2}=1 + \|\{M_j\}\|_{l^1}.
\end{align*}

\end{proof}

Our goal is to build multiresolution analyses using these functions.  To this end we recall the definition which may be found in \cite{HW}.

\begin{defn}
A \emph{multiresolution analysis} (MRA) consists of a sequence of closed subspaces $V_j, j\in\mathbb{Z}$, of $L^2(\mathbb{R}^n)$ satisfying
\begin{enumerate}
\item for all $j\in \mathbb{Z}$, $V_j\subset V_{j+1}$;
\item for all $j\in\mathbb{Z}$, $f(\cdot)\in V_j$ if and only if $f(2\cdot)\in V_{j+1}$;
\item $\displaystyle\bigcap_{j\in\mathbb{Z}}V_j= \{ 0 \}$;
\item $\displaystyle\bigcup_{j\in\mathbb{Z}}V_j$ is dense in $L^2(\mathbb{R}^n)$;
\item there exists a function $\varphi\in V_0$, such that $\{ \varphi(\cdot-j): j\in\mathbb{Z}^n  \}$ is an orthonormal basis for $V_0$.
\end{enumerate}
\end{defn}


\section{Multiresultion Analysis}

Our construction starts by considering a regular family of cardinal interpolators $\{\phi_\alpha: \alpha \in A\}$.  To each member $\phi_\alpha$, we associate the space:
\begin{equation}\label{V0}
V_0(\phi_\alpha)=\left\{ \sum_{j\in\mathbb{Z}^n}a_j L_{\phi_\alpha}(x-j): \{a_j\}\in l^2  \right\},
\end{equation}
where $L_{\phi_\alpha}$ is the fundamental function defined in \eqref{fundamental function}.
Our first result shows that this is a closed subspace of $L^2(\mathbb{R}^n)$.
\begin{prop}
Suppose that $\{ \phi_\alpha : \alpha\in A \}$ is a regular family of cardinal interpolators and $V_0(\phi_\alpha)$ is defined by \eqref{V0}, then $V_0(\phi_\alpha)$ is a subspace of $L^2(\mathbb{R}^n)$, additionally, if $\{f_k: k\in\mathbb{N}\}\subset V_0(\phi_\alpha)$ is a Cauchy sequence, the corresponding coefficient sequences form a Cauchy sequence in $l^2(\mathbb{Z}^n)$.
\end{prop}

\begin{proof}
That $V_0(\phi_\alpha)$ is a subspace follows from Theorem 1 in \cite{me}.  To see that the Cauchy criterion is shared by the coefficient sequences, we adopt the notation $\{a_j(f):j\in l^2\}$ for the coefficients of $f\in V_0(\phi_\alpha)$ and observe that
\begin{align*}
&\| f_j - f_k \|_{L^2(\mathbb{R}^n)}^2 = (2\pi)^{-n}\| \hat L_{\phi_\alpha} \sum_{l\in \mathbb{Z}^n}\left(a_l(f_j)-a_l(f_k)\right)e^{-i\langle \cdot, l \rangle}     \|_{L^2(\mathbb{R}^n)}^2\\
=& (2\pi)^{-n}\int_{\mathbb{R}^n}\left|\hat L_{\phi_\alpha}(\xi)\sum_{l\in \mathbb{Z}^n}\left(a_l(f_j)-a_l(f_k)\right)e^{-i\langle \xi, l \rangle} \right|^2d\xi \\
=& (2\pi)^{-n}\int_{[-\pi,\pi]^n} \sum_{m\in\mathbb{Z}^n}\left| \hat L_{\phi_\alpha}(\xi - 2\pi m)   \right|^2\left| \sum_{l\in \mathbb{Z}^n}\left(a_l(f_j)-a_l(f_k)\right)e^{-i\langle \xi, l \rangle}    \right|^2d\xi \\
\geq & (2\pi)^{-n}(1+C)^{-2} \int_{[-\pi,\pi]^n} \left| \sum_{l\in \mathbb{Z}^n}\left(a_l(f_j)-a_l(f_k)\right)e^{-i\langle \xi, l \rangle}    \right|^2d\xi\\
=&(1+C)^{-2}\sum_{l\in\mathbb{Z}^n}| a_l(f_j)-a_l(f_k)  |^2,
\end{align*}
where we have used Plancherel's theorem in the first line, Tonelli's theorem together with periodicity in the third, \eqref{bound} in the fourth, and finally, Parseval's theorem in the last line.  
\end{proof}

\begin{cor}
$V_0(\phi_\alpha)$ is a closed subspace of $L^2(\mathbb{R}^n)$. 
\end{cor}

Motivated by the treatment in \cite{Madych}, we introduce the scaling function $\Phi_\alpha$, which is defined by its Fourier transform,
\begin{equation}\label{scale}
\hat \Phi_\alpha (\xi) = \hat\phi_{\alpha}(\xi) \left[ \sum_{k\in\mathbb{Z}^n}\hat \phi_{\alpha}^2(\xi-2\pi k)     \right]^{-1/2},
\end{equation}
where $\xi\in \mathbb{R}^n$, and extended by continuity to the origin in the case that $\hat\phi_\alpha$ is unbounded there.  We remark that in this case we get $\hat \Phi_\alpha (0)=1$.  This turns out to be a convenient basis for $V_0(\phi_\alpha)$.

\begin{prop}
If $\Phi_\alpha$ is defined by \eqref{scale}, then $\{\Phi_\alpha(\cdot-j):j\in\mathbb{Z}^n\}$ is a complete orthonormal system in $V_0(\phi_\alpha)$.
\end{prop}

\begin{proof}
Observe that
\begin{align*}
&\langle \Phi_\alpha(\cdot-j),\Phi_\alpha(\cdot-k) \rangle = (2\pi)^{-n}\int_{\mathbb{R}^n}\hat \Phi_\alpha^2(\xi)e^{i\langle\xi , k-j \rangle} d\xi \\
=& (2\pi)^{-n} \int_{[-\pi,\pi]^n} \left(\sum_{l\in\mathbb{Z}^n} \hat \Phi_\alpha^2(\xi-2\pi l)  \right) e^{i\langle \xi,k-j \rangle} d\xi \\
=&(2\pi)^{-n} \int_{[-\pi,\pi]^n}   e^{i\langle \xi,k-j \rangle}   d\xi = \delta_{j,k},
\end{align*}
where we have used Plancherel's theorem in the first line.  The Dominated Convergence theorem allows us to switch the order of summation and integration in the second line, and the third line follows from \eqref{scale}.  Thus, orthonormality is established.  Suppose now that $f\in V_0(\phi_\alpha)$, then $\hat f = P\hat L_{\phi_\alpha}$, where $P$ is a $2\pi$ periodic function in $L^2([-\pi,\pi]^n)$.  We note that $\hat L_{\phi_\alpha} = Q \hat \Phi_\alpha$, where $Q$ is the $2\pi$ periodic function in Lemma 1.  Hence, $\hat f = PQ\hat\Phi_\alpha$, and as a result of \eqref{bound}, $PQ\in L^2([-\pi,\pi]^n)$.  This means there is an $l^2$ sequence $\{b_j:j\in\mathbb{Z}^n\}$ such that $f = \displaystyle\sum_{j\in\mathbb{Z}^n}b_j \Phi_\alpha(\cdot -j)  $, which establishes the completeness of $\{\Phi_\alpha(\cdot-j):j\in\mathbb{Z}^n\}$ in $V_0(\phi_\alpha)$.
\end{proof}

For $j\in \mathbb{Z}$, we define
\begin{equation}\label{Vj}
V_j(\phi_\alpha) = \left\{ f\in L^2(\mathbb{R}^n): f\left(2^{-j}(\cdot)\right)\in V_0(\phi_\alpha)    \right\},
\end{equation}
and
\[
V(\phi_\alpha) = \{ V_j(\phi_\alpha): j\in\mathbb{Z}   \}.
\]

Now $V(\phi_\alpha)$ is our MRA candidate, before proving our main result, we make note of the following propositions which may be found in \cite{HW}.

\begin{prop}( \cite{HW} page 45, Theorem 1.6)
Conditions 1, 2, and 5 in the definition of MRA imply condition 3.
\end{prop}

\begin{prop}( \cite{HW} page 46, Theorem 1.7)
Let $\{V_j:j\in\mathbb{Z}\}$ be a sequence of closed subspaces of $L^2(\mathbb{R}^n)$ satisfying conditions 1, 2, and 5 in the definition of MRA, assume that the scaling function $\varphi$ of condition 5 is such that $|\hat\varphi|$ is continuous at 0.  Then the following two conditions are equivalent:
\begin{itemize}
\item $\hat\varphi(0)\neq 0$, 
\item $\displaystyle \bigcup_{j\in\mathbb{Z}}V_j$ is dense in $L^2(\mathbb{R}^n)$.
\end{itemize}
Moreover, when either is the case, $|\hat\varphi(0)|=1$.
\end{prop}

We may now prove our main result.

\begin{thm}
Let $\{\phi_\alpha:\alpha\in A\}$ be a regular family of cardinal interpolators, which satisfies the condition
\[
\lim_{\xi\to\infty}\hat\phi_\alpha(\xi)=0, \qquad \text{for all } \alpha\in A.
\]
Then $V(\phi_\alpha)$ is an MRA.
\end{thm}

\begin{proof}
Conditions 1 and 2 follow immediately from \eqref{Vj}, while Proposition 2 shows that Condition 5 is satisfied.  Condition 3 now follows from Proposition 3.  To see the Condition 4 holds, we note that the limit condition implies $\hat\Phi_\alpha(0)=1$, thus we may apply Proposition 4, which finishes the proof.
\end{proof}


\section{Examples}

We present three examples of this phenomenon.  The details pertaining to regular families may be found in Section 6 of \cite{me}.

\subsection{Polyharmonic Cardinal Splines}

Our first example was explored in \cite{Madych}, and motivated this treatment.  A function or distribution $f$ is called $k$-$harmonic$, where $k\in\mathbb{N}$, if it satisfies
\begin{equation}\label{kharm}
\Delta^k f =0  \qquad\text{ on }\mathbb{R}^n,
\end{equation}    
where $\Delta$ is the Laplacian operator and  $\Delta^k( f)=\Delta(\Delta^{k-1} f)$.  A function which satisfies \eqref{kharm} with $k\geq 1$ is called \emph{polyharmonic}.  The fundamental solution of $\eqref{kharm}$ has Fourier transform given by $\hat{\phi}_k(\xi)=(2\pi)^{-n}\| \xi \|^{-2k}$, where $\|\xi\|=\sqrt{\langle \xi,\xi  \rangle}$.

A \emph{polyharmonic cardinal spline} is a function or distribution on $\mathbb{R}$ which satisfies
\begin{enumerate}
\item[(i)] $f\in C^{2k-2}(\mathbb{R}^n)$,
\item[(ii)] $\Delta^k f =0 \quad\text{on }\mathbb{R}^n\setminus\mathbb{Z}^n$.
\end{enumerate}
The family $\{\phi_k:k\in\mathbb{N}, 2k>n\}$, is a regular family of cardinal interpolators which obviously satisfies the additional hypothesis of our theorem.  Hence, for every $k\in\mathbb{N}$, $V_0(\phi_k)$ is an MRA.

\subsection{Multiquadrics I}
Our next example is a family of multiquadrics whose order is allowed to vary.  By $\alpha^{th}$ order multiquadric we mean $\phi_\alpha(x)=(\| x \|^2+c^2)^\alpha$, where $c>0$ is fixed and known as the shape parameter.  We consider the family $\{\phi_{\alpha_j}:j\in\mathbb{N}\}$, where $\{\alpha_j:j\in\mathbb{N}\}\subset [1/2,\infty)$ such that dist$\left(\{\alpha_j\},\mathbb{N}\right)>0$ and $\displaystyle\lim_{j\to\infty}\alpha_j=\infty$.  To see that the additional hypothesis in the theorem is met we note that $\hat \phi_\alpha$ is a constant multiple of $\| \cdot \|^{-\alpha - n/2 }K_{\alpha+n/2}(c\| \cdot \|)$, where $K_\beta$ is the modified Bessel function of the second kind.  Thus $\hat\phi_\alpha$ is clearly unbounded at the origin.
Before continuing, we mention the special case that $\alpha_j=j-1/2$, this family may be regarded as a family of smoothed out polyharmonic splines.

\subsection{Multiquadrics II}
Our last example is a again a family of multiquadrics, but now we fix the order and allow the shape parameter to vary.  That is, $\phi_c(x)=(\| x\|^2+c^2)^\alpha$ now we fix $\alpha\in [1/2, \infty)\setminus \mathbb{N}$ and consider the family $\{\phi_c: c\geq 1\}$.  This too is a regular family of cardinal interpolators which satisfies the additional hypothesis that $\hat\phi_c$ is unbounded at the origin. 

\subsection{Remarks}
We close by mentioning that some extra hypothesis on a regular family of cardinal interpolators is needed.  To see this, consider the family of Gaussians $\{e^{-\|x\|^2/(4\alpha)}:\alpha\geq 1\}$.  It is shown in \cite{me} that this is a regular family of cardinal interpolators.  The corresponding function $\Phi_\alpha$ clearly satisfies $\hat\Phi_\alpha(0)<1$.  Thus Proposition 4 shows that $\displaystyle\bigcup_{j\in\mathbb{Z}}V_j(e^{-\|x\|^2/(4\alpha)})$ cannot be dense in $L^2(\mathbb{R}^n)$.


\end{document}